\documentclass{amsart}
\usepackage{amsmath,amssymb,amsxtra,anysize,tikz,graphicx,tikz-cd} 

\usepackage{circuitikz}

\newcommand{\C}{\mathbb{C}}
\newcommand{\xx}{\mathbf{x}}
\newcommand{\yy}{\mathbf{y}}
\newcommand{\ttheta}{\boldsymbol{\theta}}
\newcommand{\mm}{\mathfrak{m}}
\newcommand{\DR}{\mathrm{DR}}
\newcommand{\FT}{\mathrm{FT}}
\newcommand{\HY}{\mathrm{HY}}
\newcommand{\Alt}{\mathrm{Alt}}
\newcommand{\Sym}{\mathrm{Sym}}
\renewcommand{\DH}{\mathrm{DH}}
\newcommand{\sgn}{\mathrm{sgn}}
\newcommand{\one}{\mathbf{1}}

\newcommand{\CE}{\mathcal{E}}
\newcommand{\CF}{\mathcal{F}}

\newcommand{\HH}{\mathrm{HH}}
\newcommand{\HHH}{\mathrm{HHH}}
\newcommand{\Ker}{\mathrm{Ker}}
\newcommand{\Hom}{\mathrm{Hom}}
\newcommand{\bFT}{\mathbf{FT}}
\newcommand{\hook}{\mathrm{hook}}
\newcommand{\cJ}{\mathcal{J}}
\newcommand{\cL}{\mathcal{L}}
\newcommand{\SBim}{\mathrm{SBim}}

\numberwithin{equation}{section}

\newtheorem{theorem}{Theorem}[section]
\newtheorem{proposition}[theorem]{Proposition}
\newtheorem{corollary}[theorem]{Corollary}
\newtheorem{lemma}[theorem]{Lemma}

\theoremstyle{definition}
\newtheorem{remark}[theorem]{Remark}

\newtheorem{definition}[theorem]{Definition}

\newtheorem{conjecture}[theorem]{Conjecture}

\title{Tautological classes for $(n,n+1)$ torus knots}
\author{Eugene Gorsky}
\address{Department of Mathematics, University of California Davis, One Shields Avenue, Davis CA 95616 USA}
\email{egorskiy@ucdavis.edu}

\author{Anton Mellit}
\address{Department of Mathematics, University of Vienna, Oskar-Morgenstern-Platz 1, Vienna 1090, Austria}
\email{anton.mellit@univie.ac.at}

\begin{document}

\begin{abstract}
We construct an explicit isomorphism between the HOMFLY-PT homology of $(n,n+1)$ torus knots and the direct sum of hook isotypic components of the space of diagonal coinvariants. As a consequence, we compute the action of tautological classes in   HOMFLY-PT homology of $(n,n+1)$ torus knots and prove that it extends to an action of the Lie algebra of Hamiltonian vector fields on the plane. We also compute the action of differentials $d_N$ in Rasmussen spectral sequences from HOMLFY-PT to $\mathfrak{gl}(N)$ homology of  $(n,n+1)$ torus knots. 
\end{abstract}

\maketitle

\section{Introduction}

Consider the polynomial ring $\C[x_1,\ldots,x_n,y_1,\ldots,y_n]$ with an action of $S_n$ which permutes the variables $x_i$ and $y_i$ simultaneously. The space of diagonal coinvariants $\DR_n$ is defined as 
$$
\DR_n=\C[x_1,\ldots,x_n,y_1,\ldots,y_n]/\C[\xx,\yy]^{S_n}_+
$$
where $\C[\xx,\yy]^{S_n}_+$ is the ideal generated by symmetric polynomials of positive degree. This space was introduced by Garsia and Haiman in \cite{Garsia} and has been a subject of very intensive research in algebraic combinatorics \cite{Shuffle,ShufflePf}, algebraic geometry \cite{HaimanInv,CO,Hikita} and representation theory \cite{Gordon}. In particular, the bigraded Hilbert series of its sign component $\DR_n^{\sgn}$ is known as the {\em $q,t$-Catalan polynomial} $c_n(q,t)$ \cite{Garsia}.

In \cite{G}, the first author conjectured that the dimensions of triply graded reduced Khovanov-Rozansky homology  $\overline{\HHH}$ (also known as reduced HOMFLY-PT homology) of $(n,n+1)$ torus knots are given by the $q,t$-Catalan polynomials, so we have an isomorphism of bigraded vector spaces
\begin{equation}
\label{eq: catalan intro}
\overline{\HHH}^0(T(n,n+1))\simeq \DR_n^{\sgn}.
\end{equation}
Here and below the superscript $0$ refers to zero Hochschild degree (or $a$-degree).
This conjecture was confirmed by Hogancamp in \cite{H}.
The proof in \cite{H} was based on an intricate recursion satisfied by the graded dimensions on both sides, and did not give any explicit construction of the isomorphism \eqref{eq: catalan intro}.
Similarly, one cannot directly access any homology classes in $\overline{\HHH}^0(T(n,n+1))$ besides a few very special ones. Our first main result addresses these issues.

Observe that there is a natural braid-like cobordism from the link $T(n,n)$ to the knot $T(n,n+1)$ which adds extra $(n-1)$ positive crossings. For such cobordisms, Elias and Krasner \cite{EK} defined a   map in HOMFLY homology:
$$
\pi: \HHH(T(n,n))\to \HHH(T(n,n+1)).
$$
Recall that $\HHH(T(n,n))$ is naturally a module over $\C[\xx]=\C[x_1,\ldots,x_n]$.

\begin{theorem}
\label{thm: main iso}
a) We have an isomorphism
$$
\HHH^0(T(n,n))/(\xx)\HHH^0(T(n,n))\simeq \DR_n^{\sgn}
$$
where $(\xx)=(x_1,\ldots,x_n)\subset \C[\xx]$ is the maximal ideal.

b) The cobordism map $\pi: \HHH^0(T(n,n))\to \HHH^0(T(n,n+1))$ is surjective and induces a
 grading-preserving isomorphism
\begin{equation}
\label{eq: pi prime intro}
\pi':\DR_n^{\sgn}\to \overline{\HHH}^0(T(n,n+1))
\end{equation}

c) More generally, the cobordism map $\pi: \HHH(T(n,n))\to \HHH(T(n,n+1))$ is surjective and we have an isomorphism
$$
\DR_n^{\hook}\simeq \HHH(T(n,n))/(\xx)\HHH(T(n,n))\simeq \overline{\HHH}(T(n,n+1))
$$
where $\DR_n^{\hook}$ is a direct sum of hook isotypic components in $\DR_n$, see Section \ref{sec: hooks}.
\end{theorem}

To prove Theorem \ref{thm: main iso}, we use a deformation of $\HHH$ known as the $y$-ified triply graded homology $\HY$ \cite{GH}. For knots $K$ such as $T(n,n+1)$ the $y$-ified homology is essentially the same as $\overline{\HHH}$:
\begin{equation}
\label{eq: knots intro}
\HY(K)=\HHH(K)[y]=\overline{\HHH}(K)[x,y,\theta]
\end{equation}
so we do not get any additional information. However, for the $n$-component link $T(n,n)$ the $y$-ified homology (computed in \cite{GH}) has a much richer structure, and we have the following $y$-ified version of Theorem \ref{thm: main iso}.

\begin{theorem}
\label{thm: main iso y}
a) We have an isomorphism
$$
\HY^0(T(n,n))/(\xx,\yy)\HY^0(T(n,n))\simeq \DR_n^{\sgn}
$$
where $(\xx,\yy) \subset \C[\xx,\yy]$ is the maximal ideal. More generally,
$$
\HY(T(n,n))/(\xx,\yy)\HY(T(n,n))\simeq \DR_n^{\hook}.
$$

b) There is  a map in $y$-ified homology (induced by a $y$-ified version of Elias-Krasner cobordism map) $\pi_y: \HY(T(n,n))\to \HY(T(n,n+1))$ which is surjective and induces the isomorphism \eqref{eq: pi prime intro}. 
\end{theorem}

In an earlier paper \cite{GHM} the authors together with Hogancamp defined a commuting family of operators $\CF_k,\ k\ge 1$  on $y$-ified link homology $\HY(L)$ for any link $L$ \footnote{These were denoted by $\CF_{k+1}$ in \cite{GHM}.}.  The operator $\CF_1$ satisfies a ``hard Lefshetz property" and induces an action of the Lie algebra $\mathfrak{sl}_2$ in $\HY(L)$. This also implies the existence of an involution $\Phi$ on $\HY(L)$ categorifying the $Q\leftrightarrow Q^{-1}$ symmetry of the HOMFLY-PT polynomial, and one can formally define the ``dual" operators
$$
\CE_k=\Phi\CF_k\Phi.
$$
One can use representation theory of $\mathfrak{sl}_2$ to describe the properties of $\CE_1,\CF_1$ fairly explicitly (see \cite{CG} for some examples).
However, the higher operators $\CF_k,\CE_k, k
\ge 2$ remained mysterious.

\begin{theorem}
\label{thm: main tautol}
Under the isomorphism from Theorem \ref{thm: main iso} the operators $\CF_k,\CE_k$ on $\overline{\HHH}(T(n,n+1))$ correspond to the operators $F_k=\sum_{i=1}^{n}x_i^k\partial_{y_i}, E_k=\sum_{i=1}^{n}y_i^k\partial_{x_i}$ on $\DR_n^{\hook}$.
In particular, the operators $\CF_k,\CE_k$ are nonzero if and only if $k\le n-1$.
\end{theorem}

In Theorem \ref{thm: final full} we also describe the action of differentials $d_N$ on $\overline{\HHH}(T(n,n+1))$. These appear in   Rasmussen spectral sequences \cite{Rasmussen} from $\overline{\HHH}$ to $\mathfrak{gl}(N)$ Khovanov-Rozansky homology, and an explicit description for their action in homology of torus knots $T(m,n)$ has been conjectured in \cite{GORS}. Theorem \ref{thm: final full} confirms this conjecture for $m=n+1$. It is known  that the operators $\CF_k$ commute with the differentials $d_N$ for all $k,N$.

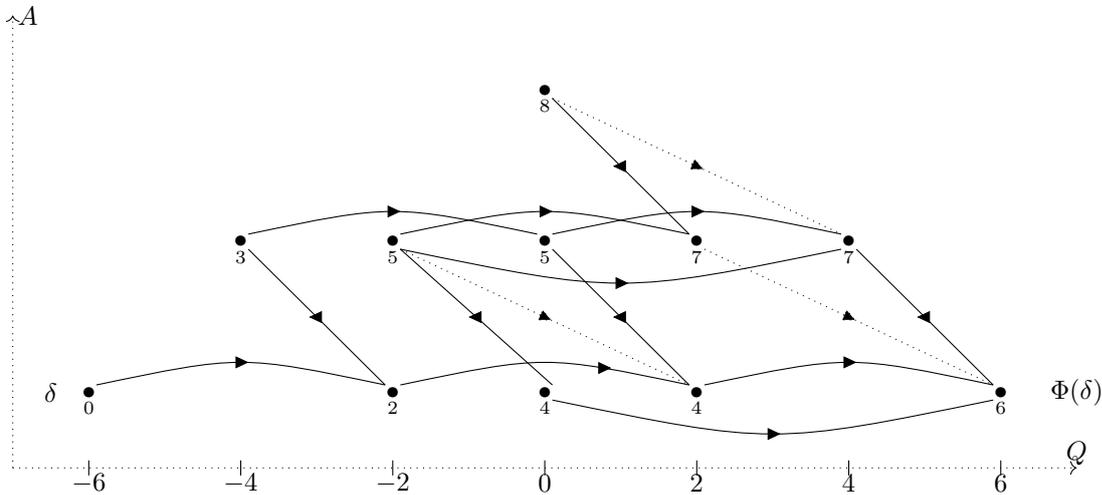
\begin{figure}[ht!]
\begin{tikzpicture}[scale=1]
\draw[dotted,->](-7,-1)--(7,-1);
\draw[dotted,->](-7,-1)--(-7,5);
\draw (7,-0.8) node {$Q$};
\draw (-6.8,5) node {$A$};
\draw (-6,-1.1)--(-6,-0.9);
\draw (-6,-1.2) node {$-6$};
\draw (-4,-1.1)--(-4,-0.9);
\draw (-4,-1.2) node {$-4$};
\draw (-2,-1.1)--(-2,-0.9);
\draw (-2,-1.2) node {$-2$};
\draw (0,-1.1)--(0,-0.9);
\draw (0,-1.2) node {$0$};
\draw (2,-1.1)--(2,-0.9);
\draw (2,-1.2) node {$2$};
\draw (4,-1.1)--(4,-0.9);
\draw (4,-1.2) node {$4$};
\draw (6,-1.1)--(6,-0.9);
\draw (6,-1.2) node {$6$};

\draw(-6,-0.2) node {\scriptsize $0$};
\draw(-2,-0.2) node {\scriptsize $2$};
\draw(0,-0.2) node {\scriptsize $4$};
\draw(2,-0.2) node {\scriptsize $4$};
\draw(6,-0.2) node {\scriptsize $6$};

\draw(-4,1.8) node {\scriptsize $3$};
\draw(-2,1.8) node {\scriptsize $5$};
\draw(0,1.8) node {\scriptsize $5$};
\draw(2,1.8) node {\scriptsize $7$};
\draw(4,1.8) node {\scriptsize $7$};

\draw(0,3.8) node {\scriptsize $8$};

\draw (-6.5,0) node {$\delta$};
\draw (-6,0) node {$\bullet$};
\draw (-2,0) node {$\bullet$};
\draw (0,0) node {$\bullet$};
\draw (2,0) node {$\bullet$};
\draw (6,0) node {$\bullet$};
\draw (7,0) node {$\Phi(\delta)$};
\draw (-4,2) node {$\bullet$};
\draw (-2,2) node {$\bullet$};
\draw (-0,2) node {$\bullet$};
\draw (2,2) node {$\bullet$};
\draw (4,2) node {$\bullet$};
\draw (0,4) node {$\bullet$};
\draw (-5.9,0.1)..controls (-4,0.5) and (-4,0.5)..(-2.1,0.1) node[currarrow,pos = 0.5]{};
\draw  (-1.9,0.1)..controls (0,0.5) and (0,0.5)..(1.9,0.1) node[currarrow,pos = 0.75]{};
\draw (2.1,0.1)..controls (4,0.5) and (4,0.5)..(5.9,0.1) node[currarrow,pos = 0.5]{};
\draw  (-3.9,2.1)..controls (-2,2.5) and (-2,2.5)..(-0.1,2.1) node[currarrow,pos = 0.5]{};
\draw  (0.1,2.1)..controls (2,2.5) and (2,2.5)..(3.9,2.1) node[currarrow,pos = 0.5]{};
\draw  (-1.9,2.1)..controls (0,2.5) and (0,2.5)..(1.9,2.1) node[currarrow,pos = 0.5]{};

\draw  (0.1,-0.1)..controls (3,-0.7) and (3,-0.7)..(5.9,-0.1) node[currarrow,pos = 0.5]{};

\draw  (-1.9,1.9)..controls (1,1.3) and (1,1.3)..(3.9,1.9) node[currarrow,pos = 0.5]{};

\draw  (-2.1,0.1)--(-3.9,1.9) node[currarrow,pos = 0.5,xscale=-1]{};
\draw (-1.9,1.9)--(0.1,0.1) node[currarrow,pos = 0.5,xscale=-1]{};
\draw (0.1,1.9)--(1.9,0.1)node[currarrow,pos = 0.5,xscale=-1]{};
\draw (4.1,1.9)--(5.9,0.1)node[currarrow,pos = 0.5,xscale=-1]{};
\draw (0.1,3.9)--(1.9,2.1)node[currarrow,pos = 0.5,xscale=-1]{};
\draw [dotted] (0.1,3.9)--(3.9,2.1)node[currarrow,pos = 0.5,sloped]{};
\draw [dotted] (2.1,1.9)--(5.9,0.1)node[currarrow,pos = 0.5,sloped]{};
\draw [dotted] (-1.9,1.9)--(1.9,0.1)node[currarrow,pos = 0.5,sloped]{};
\end{tikzpicture}
\caption{Operators $\CF_k$ and differentials for $T(3,4)$: $\CF_1$ and $\CF_2$  correspond to shorter and longer horizontal arrows; $d_1$ and $d_2$ correspond to solid and dotted diagonal arrows.}
\label{Fig: T34}
\end{figure}

We illustrate the action of $\CF_k$ and $d_N$ for $n=3$ in Figure \ref{Fig: T34} (see also Figure \ref{Fig: catalan} below for a more detailed picture of $\overline{\HHH}^0$). The gradings in Figure \ref{Fig: T34} match the conventions of \cite{DGR}, with the axes corresponding to the $Q$ and $A$ gradings. The $T$-gradings are indicated next to the generators. The operators $\CF_1$ and $\CF_2$ preserve the $A$-grading and correspond to shorter and longer horizontal arrows. The differentials $d_1$ and $d_2$ decrease the $A$-grading and correspond to solid and dotted diagonal arrows.

As an application of deep results of Haiman \cite{HaimanInv}, we prove that the operators $\CF_k$ and the differentials $d_N$ ``co-generate" $\overline{\HHH}(T(n,n+1))$ in the following sense. Recall that for any positive braid $\beta$ the homology $\HHH(\beta)$ has a distinguished generator $\delta_{\beta}$ corresponding to the rightmost copy of $R$ in the Rouquier complex (see Section \ref{sec: background} for all details). The generator $\delta=\delta_{\beta}$ for $T(3,4)$ is shown on the left of Figure \ref{Fig: T34}. Under the isomorphism from Theorem \ref{thm: main iso}, $\delta$ corresponds to the Vandermonde determinant $\Delta(\yy)$ while $\Phi(\delta)$ corresponds to $\Delta(\xx)$.

\begin{corollary}
\label{cor: main cogenerate}
Let $\delta$ be the rightmost generator for $T(n,n+1)$, and $f$ an arbitrary nonzero homogeneous generator in $\overline{\HHH}(T(n,n+1))$. Then there exists an element $\varphi\in \C[\CF_1,\ldots,\CF_{n-1}]\otimes \wedge^{\bullet}\langle d_1,\ldots,d_{n-1}\rangle$ such that $\varphi(f)=c\Phi(\delta)$ for some $c\neq 0$. 
\end{corollary}

Next, we use Theorem \ref{thm: main tautol} to study the commutation relations between $\CE_k$ and $\CF_k$. Given a polynomial $H\in \C[x,y]$, the corresponding {\em Hamiltonian vector field} $v_H$ is the first order differential operator on $\C[x,y]$ defined as 
$$
v_H=\frac{\partial H}{\partial x}\partial_y-\frac{\partial H}{\partial y}\partial_x.
$$
For example,
$$
v_{x^k}=kx^{k-1}\partial_y,\ v_{y^k}=-ky^{k-1}\partial_x.
$$
The Hamiltonian vector fields generate a Lie algebra $\mathcal{H}_2$, and we can consider a subalgebra $\mathcal{H}_2^{\ge 2}$ spanned by the vector fields $v_{x^ay^b}$ with $a+b\ge 2$. Theorem \ref{thm: main tautol} implies the following. 

\begin{corollary}
\label{cor: main Hamiltonian}
The action of $\CE_k,\CF_k$ on $\overline{\HHH}(T(n,n+1))$ extends to the action of the Lie algebra $\mathcal{H}_2^{\ge 2}$. 
\end{corollary}

In fact, inspired by recent work of the second author \cite{HMMS,ICM}, we conjecture that this is a general phenomenon which works for all links.

\begin{conjecture}
\label{conj: Hamiltonian}
Given an arbitrary link $L$, there is an action of the Lie algebra $\mathcal{H}_2$ on $\HY(L)$ extending the action of $\CE_k,\CF_k$. If $L$ is a knot then there is an action of $\mathcal{H}_2^{\ge 2}$ on $\overline{\HHH}(L)$. 
\end{conjecture}

\begin{remark}
The second statement follows from the first by \eqref{eq: knots intro} and Lemma \ref{lem: reduced}. 
\end{remark}

The paper is organized as follows. In Section \ref{sec: background} we recall some background on link homology, $y$-ification and tautological classes following \cite{GH,GHM}. We also recall some facts about the space $\DR_n$ of diagonal harmonics and its sign component $\DR_n^{\sgn}$, in particular the ``Operator Conjecture" about $F_k$ proved by Haiman in \cite{HaimanInv} (see Theorem \ref{thm: operator conjecture}). Finally, we briefly discuss the Lie algebra of Hamiltonian vector fields. We do not assume the reader's familiarity with diagonal harmonics or Hamiltonian vector fields, and provide proofs of many known results to clarify the exposition.

In Section \ref{sec:links} we first recall the description of $\HY(T(n,n))$ from \cite{GH} and use it to compute the action of tautological classes for $T(n,n)$. Then we construct an explicit map $\pi_y$ relating $\HY(T(n,n))$ to $\HY(T(n,n+1))$ and establish its properties. In particular, we prove that it commutes with the action of tautological classes as in Theorem \ref{thm: main tautol}. Finally, in Theorem \ref{thm: final} we use all of the above results to prove that it factors through $\DR_n^{\sgn}$ and  yields a desired isomorphism $\pi'$. This proves Theorems \ref{thm: main iso y} and \ref{thm: main tautol} for $a=0$. In Section \ref{sec: hooks} we generalize these arguments to higher $a$-degrees in Theorem \ref{thm: final full}.

Theorem \ref{thm: main iso} is a consequence of Theorem \ref{thm: main iso y} and the isomorphism 
$$
\HHH(T(n,n))=\HY(T(n,n))/(\yy)\HY(T(n,n))
$$
proved in \cite{GH} (see Theorem \ref{thm: GH}).  Finally, Corollary \ref{cor: main Hamiltonian} follows from the description of $E_k,F_k$ on $\DR_n$ and 
Corollary \ref{cor: main cogenerate} follows from the above and Lemma \ref{lem: cogenerate} (for $a=0$) and Lemma \ref{lem: hooks generate} (in general).

\section*{Acknowledgments}

We thank Matt Hogancamp and Lev Rozansky for useful discussions related to Conjecture \ref{conj: Hamiltonian}.

The work of E. G. was partially supported by the NSF grant DMS-2302305. The research of A. M. is supported by the Consolidator Grant No. 101001159 ``Refined invariants in combinatorics, low-dimensional topology and geometry of moduli spaces” of the
European Research Council.

\section{Background}
\label{sec: background}

\subsection{Link homology}

We recall some definitions on link homology and refer to \cite{GH,GHM} for more details and grading conventions. Let $R=\C[x_1,\ldots,x_n]$. Consider the $R-R$ bimodules
$$
B_i=\frac{\C[x_1,\ldots,x_n,x'_1,\ldots,x'_n]}{x_i+x_{i+1}=x'_i+x'_{i+1},x_ix_{i+1}=x'_ix'_{i+1},x_j=x'_j\ (j\neq i,i+1)}.
$$
The category of Soergel bimodules $\SBim_n$ is defined as the smallest subcategory of $R-R$-bimodules containing $R$ and $B_i$ and closed under direct sums, direct summands, grading shifts and tensor products. 

There exist bimodule maps $b_i:B_i\to R$ and $b_i^*:R\to B_i$ such that 
$$
b_ib^*_i=b^*_ib_{i}=x_i-x'_{i+1}.
$$
The {\em Rouquier complexes} are defined as the cones of these maps:
$$
T_i=\left[B_i\xrightarrow{b_i} R\right],\ T_i^{-1}=\left[R\xrightarrow{b_i^*} B_i\right].
$$
\begin{theorem}(\cite{Rouquier})
The complexes $T_i,T_i^{-1}$ satisfy braid relations up to homotopy. 
\end{theorem}
This allows one to define a complex of Soergel bimodules $T(\beta)$ for any $n$-strand braid $\beta$. The HOMFLY-PT homology of $\beta$ is defined by first applying the Hochschild homology to each term of this complex, and then computing the homology of the result:
$$
\HHH(\beta)=H^*(\HH(T(\beta)).
$$
\begin{theorem}(\cite{KhSoergel,KR2})
The homology $\HHH(\beta)$ is a topological invariant of the link $L$ obtained by the closure of $\beta$.
\end{theorem}

Next, we define a deformation, or $y$-ification of the Rouquier complexes. We tensor all Soergel bimodules by $\C[\yy]$. The $y$-ified crossings have the form
$$
T_i^y=\left[
\begin{tikzcd}
B_i[\yy]\arrow[bend left]{r}{b_i} & R[\yy] \arrow[bend left]{l}{b_i^*(y_i-y_{i+1})}
\end{tikzcd}
\right],\ (T_i^y)^{-1}=\left[
\begin{tikzcd}
R[\yy]\arrow[bend left]{r}{b_i^*} & B_i[\yy] \arrow[bend left]{l}{b_i(y_i-y_{i+1})}
\end{tikzcd}
\right]
$$
These are curved complexes where the  differential  does not square to zero but satisfies $$D^2=(x_i-x'_{i+1})(y_i-y_{i+1}).$$

\begin{theorem}(\cite{GH})
The complexes $T_i^y,(T_i^y)^{-1}$ satisfy braid relations up to homotopy. For any braid $\beta$ and the corresponding permutation $w$ there is a well-defined (up to homotopy) curved complex $T^y(\beta)$ with 
\begin{equation}
\label{eq: curvature}
D^2=\sum_{i=1}^{n}y_i(x_i-x'_{w(i)}).
\end{equation}
\end{theorem}

Note that the identity braid $\one$ corresponds to $T^y(\one)=R[\yy]$ which we will also denote by $\one^y$.
We will also need curved Koszul complexes 
$$
K_{ij}^y=\left[
\begin{tikzcd}
R[\yy]\arrow[bend left]{r}{x_i-x_j} & R[\yy] \arrow[bend left]{l}{y_i-y_j}
\end{tikzcd}
\right],\ K_i^y=K_{i,i+1}^y.
$$
Given a braid $\beta$ with the corresponding permutation $w$, consider a minimal length factorization of $w^{-1}$ into (not necessary simple) transpositions:
$$
w^{-1}=(i_1 j_1)\cdots (i_s j_s).
$$
Then the $y$-ified homology of $\beta$ is defined as
\begin{equation}
\label{eq: def HY}
\HY(\beta)=H^*\left(\HH\left(T^y(\beta)\otimes K^y_{i_1,j_1}\otimes \cdots \otimes K_{i_s,j_s}^y\right)\right).
\end{equation}
One can check that after applying $\HH$ the differential indeed squares to zero, so we can define the homology of the resulting complex. Also, one can check that the result does not depend on factorization of $w$ into transpositions.

\begin{theorem}(\cite{GH})
The homology $\HY(\beta)$ is a topological invariant of the link $L$ obtained by the closure of $\beta$.
\end{theorem}

We will need the following facts.

\begin{proposition}
\label{prop: skein}
There is a chain map $\psi:T_i^y\to (T_i^y)^{-1}$. The cone of $\psi$ is homotopy equivalent to the Koszul complex $K_i^y$.
\end{proposition}

\begin{corollary}
Suppose that $\beta_+=\beta_1\sigma_i\beta_2$ and $\beta_-=\beta_1\sigma_i^{-1}\beta_2$.  Then there are maps in $y$-ified homology
\begin{multline}
\label{eq: skein maps}
\psi_{+-}: \HY(\beta_+)\to \HY(\beta_-),\\ \psi_{-0}:\HY(\beta_-)\to \HY(T^y(\beta_1)K_i^yT^y(\beta_2)),\\  \psi_{0+}:\HY(T^y(\beta_1)K_i^yT^y(\beta_2))\to \HY(\beta_+) 
\end{multline}
which fit into a long exact sequence. 
\end{corollary}

In \cite{GHM} the authors together with Hogancamp defined a family of operators $\CF_k$ on $y$-ified link homology. We summarize some of the properties of $\CF_k$ in the following theorem.

\begin{theorem}(\cite{GHM})
\label{thm: Fk summary}
The operators $\CF_k$ satisfy the following properties:
\begin{itemize}
\item[(a)] The space $\HY(L)$, together with the action of $\CF_k$  is a topological invariant of the link $L$.
\item[(b)] The operators $\CF_k$ pairwise commute: $[\CF_k,\CF_m]=0$.
\item[(c)] The commutation relations between $\CF_k,x_i$ and $y_i$ are given by $[\CF_k,x_i]=0$, $[\CF_k,y_i]=x_i^k$.
\item[(d)] The maps \eqref{eq: skein maps} in the skein exact triangle commute with the action of $\CF_k$.
\item[(e)] There exists an operator $\CE_1$ such that $(\CE_1,\CF_1)$ generate an action of the Lie algebra $\mathfrak{sl}_2$ on $\HY(L)$.  
\item[(f)] There exists an automorphism $\Phi:\HY(L)\to \HY(L)$ such that $\Phi \CE_1=\CF_1\Phi$ and $\Phi x_i=y_i\Phi$. Furthermore, $\Phi^2=\mathrm{Id}$. 
\end{itemize}
\end{theorem}

Parts (e) and (f) of Theorem \ref{thm: Fk summary} follow from the ``curious hard Lefshetz" property of the operator $\CF_1$, namely that
$$
\CF_1^{j}:\HY^{i}_{-2j,k}\to \HY^{i}_{2j,k+2j}
$$
is an isomorphism for all $i,k$ and all $j\ge 0$. One can then reconstruct the actions of $\CE_1$ and of $\Phi$ using representation theory of $\mathfrak{sl}_2$. Note that while $\HY(L)$ is infinite-dimensional, it decomposes as a direct sum of finite-dimensional representations of $\mathfrak{sl}_2$.

A nonzero vector $v\in \HY^{i}_{-2j,k}$  is called a {\em highest weight vector} if $j\ge 0$ and $\CF_1^{j+1}v=0$. One can prove (see for example \cite[Lemma 7.2]{GHM}) that ``curious hard Lefshetz" property implies that $\HY$ has a basis of the form $\CF_1^sv$ where $v\in \HY^{i}_{-2j,k}$ is a highest weight vector and $s\le j$. The operator $\CE_1$ is defined in this basis by
\begin{equation}
\label{eq: def E1}
\CE_1(\CF_1^sv)=\begin{cases}
0 & \mathrm{if}\ s=0\\
s(j-s+1)\CF_1^{s-1}v & \mathrm{if}\ s>0.
\end{cases}
\end{equation}
The operator $\Phi$ is defined by
\begin{equation}
\label{eq: def Phi}
\Phi(\CF_1^sv)=\begin{cases}
\frac{1}{(j-s)\cdots (s+1)}\CF_1^{j-s}v & \mathrm{if}\ j-2s>0\\
\CF_1^sv & \mathrm{if}\ j-2s=0\\
s\cdots (j-s+1)\CF_1^{j-s}v & \mathrm{if}\ j-2s<0.
\end{cases}
\end{equation}
We can use the action of $\Phi$ to define more operators.
\begin{definition}
\label{def: Ek homology}
We define operators $\CE_k=\Phi \CF_k \Phi$ for all $k$.
\end{definition}

It is worth to emphasize that while the operators $\CF_k$ are constructed in \cite{GHM} on the chain level, the operators $\CE_k$ (or even $\CE_1$) are only known on the level of homology and are rather mysterious.
Nevertheless, we have the following.

\begin{corollary}
\label{cor: hard Lefshetz}
Suppose that $W:\HY(L)\to \HY(L')$ is a grading-preserving map between the $y$-ified homologies of two links $L,L'$ commuting with the action of $\CF_k$ for all $k$. Then $W$ commutes with the action of $\Phi$ and of dual operators $\CE_k$.
\end{corollary}

\begin{proof}
Since $W$ commutes with $\CF_1$ and preserves gradings, it sends highest weight vectors to highest weight vectors. By \eqref{eq: def E1} and \eqref{eq: def Phi} we conclude that $W$ commutes with $\CE_1$ and with $\Phi$, and the result follows.
\end{proof}

Next, we briefly discuss the differentials $d_N$. These were introduced in \cite{Rasmussen} as the first nontrivial differentials in the spectral sequences from $\HHH$ to $\mathfrak{gl}(N)$ Khovanov-Rozansky homology. We refer to \cite[Section 3]{BGHW} for more details and an equivalent construction of $d_N$. 
For the purposes of this paper, we will only need the following results:

\begin{proposition}\cite[Proposition 3.18]{BGHW}
\label{prop: dN functorial}
a) The action of $d_N$ commutes with arbitrary   maps of Soergel bimodules, and with the chain maps of complexes of Soergel bimodules.

b)In particular, $d_N$ commutes with tautological classes $\CF_k$ for all $k$. 
\end{proposition}

\subsection{Diagonal harmonics}

We recall some results of Haiman with \cite{HaimanInv} as the main reference. 

Consider the ring $\C[\xx,\yy]=\C[x_1,\ldots,x_n,y_1,\ldots,y_n]$ as above, it has an action of $S_n$ which permutes $x_i$ and $y_i$ simultaneously. Let $\C[\xx,\yy]^{S_n}_{+}$ be the ideal in $\C[\xx,\yy]$ generated by homogeneous symmetric polynomials of positive degree. 

\begin{definition}
The diagonal coinvariant ring is defined by 
$$
\DR_n=\C[\xx,\yy]/\C[\xx,\yy]^{S_n}_{+}
$$
\end{definition}

\begin{theorem}[\cite{HaimanInv}]
We have $\dim \DR_n=(n+1)^{n-1}$ and $\dim \DR_n^{\sgn}=c_n$ where 
$$
c_n=\frac{1}{n+1}\binom{2n}{n}
$$
is the $n$th Catalan number.
\end{theorem}

Let $J$ be the ideal in $\C[\xx,\yy]$ generated by all antisymmetric polynomials, and let $\mm=(\xx,\yy)$ be the maximal ideal in $\C[\xx,\yy]$. The following is well known but we include a  proof for the reader's convenience. 

\begin{lemma}
\label{lem: catalan}
We have $\DR_n^{\sgn}\simeq J/\mm J$.
\end{lemma}

\begin{proof}
Let $L\subset \C[\xx,\yy]$ be the subspace spanned by the products $f(\xx,\yy)g(\xx,\yy)$ where $f$ and $g$ are homogeneous, $f$ has positive degree and symmetric, and $g$ is antisymmetric. Let 
$$
\Alt=\frac{1}{n!}\sum_{\sigma\in S_n}\sgn(\sigma)\sigma,\quad 
\Sym=\frac{1}{n!}\sum_{\sigma\in S_n}\sigma
$$
be the antisymmetrization and symmetrization operators in $\C[S_n]$ respectively.

Observe that any polynomial in $L$ is antisymmetric, and $L\subset \C[\xx,\yy]^{S_n}_{+}\cap \mm J$.  We claim that in fact $L$ is the sign isotypic component of both $\C[\xx,\yy]^{S_n}_{+}$ and $\mm J$. 
Indeed, any element in $\C[\xx,\yy]^{S_n}_{+}$ can be written as $\sum f_ih_i$ where $f_i$ are symmetric polynomials of positive degree. Then $\Alt(\sum f_ih_i)=\sum f_i\Alt(h_i)\in L$.
Similarly, any polynomial in $\mm J$ can be written as 
$\sum g_i h_i$ where $g_i$ are antisymmetric polynomials and $h_i$ are some homogeneous polynomials of positive degree. Then $\Alt(\sum g_ih_i)=\sum g_i\Sym(h_i)\in L$.

Since $J$ is generated by antisymmetric polynomials, we have
$$
J/\mm J=\C[\xx,\yy]^{\sgn}/(\mm J)^{\sgn}=\C[\xx,\yy]^{\sgn}/L=\DR_n^{\sgn}.
$$
\end{proof}

Next, we would like to utilize the bilinear pairing on $\C[\xx,\yy]$ defined by 
$$
\langle f,g\rangle f(\partial_{\xx},\partial_{\yy})g(\xx,\yy)|_{\xx=\yy=0}.
$$
It is well-known that this pairing is symmetric, nondegenerate, $S_n$-invariant and respects the bigrading.  

\begin{definition}
The space of diagonal harmonics is defined by 
$$
\DH_n=\left\{f\in \C[\xx,\yy]\mid g(\partial_{\xx},\partial_{\yy})f=0\ \forall g\in \C[\xx,\yy]^{S_n}_{+}.\right\}
$$
\end{definition}
It is easy to see that $\DH_n$ is the graded dual space to $\DR_n$ with respect to the pairing, and similarly $\DH_n^{\sgn}$ is graded dual to $\DR_n^{\sgn}$.

Next, we define a family of differential operators on $\DR_n$ and $\DH_n$. 

\begin{definition}
We define 
\begin{equation}
\label{eq: def E algebra}
E_k^*=\sum_{i=1}^{n}x_i\partial_{y_i}^k,\ E_k=\sum_{i=1}^{n}y_i^k\partial_{x_i}
\end{equation}
and
\begin{equation}
\label{eq: def F algebra}
F_k^*=\sum_{i=1}^{n}y_i\partial_{x_i}^k,\ F_k=\sum_{i=1}^{n}x_i^k\partial_{y_i}
\end{equation}
\end{definition}
Clearly, $E_k$ and $E_k^*$ (resp. $F_k$ and $F_k^*$) are adjoint with respect to the pairing up to some scalar which we will not need.
Note that $E_k^*$ pairwise commute and commute with $\partial_{y_i}$ (resp. $F_k^*$ commute with $\partial_{x_i}$) for all $i$. It is known that $E_k^*$ and $F_k^*$ preserve the subspace $\DH_n$ while $E_k$ and $F_k$ preserve the ideal $\C[\xx,\yy]^{S_n}_+$ (compare with Lemma \ref{lem: Ek Fk preserve} below) and hence are well defined on $\DR_n$.

\begin{theorem}(\cite[Theorem 4.2]{HaimanInv})
\label{thm: operator conjecture}
Let $$\Delta(\xx)=\prod_{i<j}(x_i-x_j), \Delta(\yy)=\prod_{i<j}(y_i-y_j). $$ Then $\Delta(\xx)$ generates the space of diagonal harmonics $\DH_n$ under the action of $\C[F^*_1,\ldots,F_{n-1}^*,\partial_{x_1},\ldots,\partial_{x_n}]$.
Similarly,
$\Delta(\yy)$ generates $\DH_n$ under the action of $\C[E_1^*,\ldots,E_{n-1}^*,\partial_{y_1},\ldots,\partial_{y_n}]$. 
\end{theorem}

\begin{remark}
The paper \cite{HaimanInv} proved only the first statement. However, since all constructions are symmetric in $\xx$ and $\yy$, the second statement follows immediately.
\end{remark}

\begin{remark}
Since $\Delta(\yy)$  and $\Delta(\xx)$ have degree $n-1$ in each individual variable, we get $E_k^*\Delta(\yy)=0$ and $F_k^*\Delta(\xx)=0$ for $k\ge n$. Theorem \ref{thm: operator conjecture} then implies that $E_k^*$ and $F_k^*$ are identically zero on $\DH_n$ for $k\ge n$.
\end{remark}

Note that $E_i^*$ and $F_i^*$ are $S_n$-equivariant, and in particular preserve the subspace $\DH_n^{\sgn}$.

\begin{corollary}
\label{cor: operator conjecture sgn}
The polynomial $\Delta(\xx)$ generates  $\DH_n^{\sgn}$ under the action of $\C[F^*_1,\ldots,F_{n-1}^*]$ and 
$\Delta(\yy)$ generates $\DH_n$ under the action of $\C[E_1^*,\ldots,E_{n-1}^*]$. 
\end{corollary}

\begin{proof}
We prove the first statement, the second is similar. By Theorem \ref{thm: operator conjecture} we can write any element of $\DH_n$ as $f=\sum_i f_i(F^*_1,\ldots,F_{n-1}^*)g_i(\partial_{x_1},\ldots,\partial_{x_n})\Delta(\xx)$. If $\widetilde{g_i}=\Sym(g_i)$ then (since $F_i^*$ are $S_n$-invariant and $\Delta(\xx)$ is anti-invariant)
$$
\Alt(f)=\sum_i f_i(F^*_1,\ldots,F_{n-1}^*)\widetilde{g_i}(\partial_{x_1},\ldots,\partial_{x_n})\Delta(\xx).
$$
But $\widetilde{g_i}(\partial_{x_1},\ldots,\partial_{x_n})\Delta(\xx)=0$ for a symmetric polynomial $\widetilde{g_i}$ unless $\widetilde{g_i}$ is a constant, and the result follows.
\end{proof}

\begin{lemma}
\label{lem: Ek Fk preserve}
The operators $E_k$ and $F_k$ preserve the ideals $J$ and $\mm J$. 
\end{lemma}

\begin{proof}
Any element of $J$ can be written as $\sum_i f_ig_i$ where $f_i$ are arbitrary and $g_i$ are antisymmetric. Then
\begin{equation}
\label{eq: Ek Leibniz}
E_k\left(\sum_i f_ig_i\right)=\sum_i E_k(f_i)g_i+\sum_i f_iE_k(g_i).
\end{equation}
Now $E_k(g_i)$ are antisymmetric, so all the terms in \eqref{eq: Ek Leibniz} are in $J$ and the first statement follows. Note that here we used the fact that $E_k$ are first order differential operators.

For the second statement, assume that $f_i$ are homogeneous of total degrees $d_i>0$, so that $\sum_i f_ig_i\in \mm J$. Then $E_k(f_i)$ has total degree $k+d_i-1>0$, so all terms in \eqref{eq: Ek Leibniz} are in $\mm J$.
\end{proof}

As a consequence, the operators $E_k$ and $F_k$ are well defined on the quotient $J/\mm J$, and it is clear from definitions that the isomorphism from Lemma \ref{lem: catalan} agrees with their respective actions.
The following result is a dual statement to Corollary \ref{cor: operator conjecture sgn}. It shows that $\Delta(\xx)$ and $\Delta(\yy)$ ``co-generate" the space $\DR_n^{\sgn}$ under the action of $E_k$ and $F_k$.

\begin{lemma}
\label{lem: cogenerate}
Let $f\in \DR_n^{\sgn}\simeq J/\mm J$ be a nonzero homogeneous element. Then there exists a polynomial $\varphi(F_1,\ldots,F_{n-1})$ such that  $\varphi(F_1,\ldots,F_{n-1})f$ is a nonzero multiple of $\Delta(\xx)$. Similarly, there exists a polynomial $\psi(E_1,\ldots,E_{n-1})$ such that  $\psi(E_1,\ldots,E_{n-1})f$ is a nonzero multiple of $\Delta(\yy)$.
\end{lemma}

\begin{proof}
We prove the first statement, the second is similar. It is well known that homogeneous component of $\DR_n$ of bidegree $q^{n(n-1)/2}t^0$ is one-dimensional and spanned by $\Delta(\xx)$. For the sake of contradiction, assume that for all $\varphi$ the coefficient at $\Delta(x)$ actually vanishes, and $\langle \varphi(F_1,\ldots,F_{n-1})f,\Delta(x)\rangle=0$. Then by adjunction we get
$$
\langle f, \varphi(F_1^*,\ldots,F_{n-1}^*)\Delta(x)\rangle=0
$$
which contradicts Corollary \ref{cor: operator conjecture sgn}.
\end{proof}

We illustrate the above lemmas for $n=3$. The spaces $(\DH_n)^{\sgn}$ and $\DR_n^{\sgn}$ are 5-dimensional, and the bigradings and action of operators $F^*_k,F_k$ is given by Figure \ref{Fig: catalan}. By transposing the figures, one can get the action of $E_k,E^*_k$.

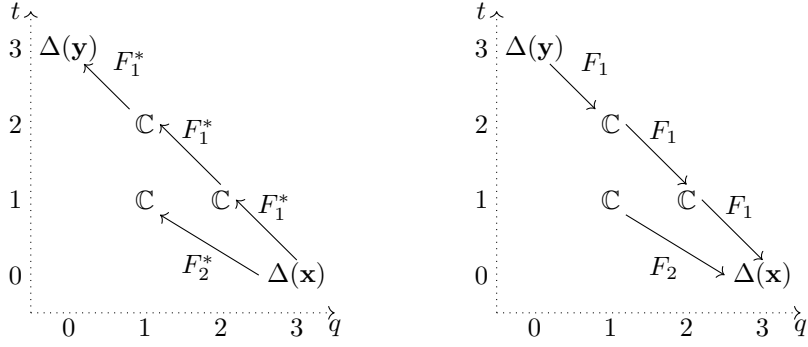
\begin{figure}[ht!]
\begin{tikzpicture}
\draw [dotted,->] (0,0)--(4,0);
\draw [dotted,->] (0,0)--(0,4);
\draw (0.5,3.5) node {$\Delta(\yy)$};
\draw (1.5,2.5) node {$\C$};
\draw (2.5,1.5) node {$\C$};
\draw (3.5,0.5) node {$\Delta(\xx)$};
\draw (1.5,1.5) node {$\C$};

\draw [->] (3,0.5)--(1.7,1.3);
\draw [->] (3.5,0.7)--(2.7,1.5);
\draw [->] (2.5,1.7)--(1.7,2.5);
\draw [->] (1.3,2.7)--(0.7,3.3);
\draw (2.2,0.6) node {$F_2^*$};
\draw (3.2,1.4) node {$F_1^*$};
\draw (2.2,2.4) node {$F_1^*$};
\draw (1.3,3.3) node {$F_1^*$};

\draw (4,-0.2) node {$q$};
\draw (-0.2,4) node {$t$}; 
\draw (0.5,-0.2) node {$0$};
\draw (1.5,-0.2) node {$1$};
\draw (2.5,-0.2) node {$2$};
\draw (3.5,-0.2) node {$3$};
\draw (-0.2,0.5) node {$0$};
\draw (-0.2,1.5) node {$1$};
\draw (-0.2,2.5) node {$2$};
\draw (-0.2,3.5) node {$3$};
\end{tikzpicture}
\qquad
\qquad
\begin{tikzpicture}
\draw [dotted,->] (0,0)--(4,0);
\draw [dotted,->] (0,0)--(0,4);
\draw (0.5,3.5) node {$\Delta(\yy)$};
\draw (1.5,2.5) node {$\C$};
\draw (2.5,1.5) node {$\C$};
\draw (3.5,0.5) node {$\Delta(\xx)$};
\draw (1.5,1.5) node {$\C$};

\draw [<-] (3,0.5)--(1.7,1.3);
\draw [<-] (3.5,0.7)--(2.7,1.5);
\draw [<-] (2.5,1.7)--(1.7,2.5);
\draw [<-] (1.3,2.7)--(0.7,3.3);
\draw (2.2,0.6) node {$F_2$};
\draw (3.2,1.4) node {$F_1$};
\draw (2.2,2.4) node {$F_1$};
\draw (1.3,3.3) node {$F_1$};

\draw (4,-0.2) node {$q$};
\draw (-0.2,4) node {$t$}; 
\draw (0.5,-0.2) node {$0$};
\draw (1.5,-0.2) node {$1$};
\draw (2.5,-0.2) node {$2$};
\draw (3.5,-0.2) node {$3$};
\draw (-0.2,0.5) node {$0$};
\draw (-0.2,1.5) node {$1$};
\draw (-0.2,2.5) node {$2$};
\draw (-0.2,3.5) node {$3$};

\end{tikzpicture}
\caption{Action of $F_k^*$ on $\DH_3^{\sgn}$ (left) and action of $F_k$ on $\DR_3^{\sgn}$ (right)}
\label{Fig: catalan}
\end{figure}

\subsection{Hamiltonian vector fields}

The algebra $\C[x,y]$ of polynomials in two variables has a Poisson bracket 
$$
\{f,g\}=\frac{\partial f}{\partial x}\frac{\partial g}{\partial y}-\frac{\partial f}{\partial y}\frac{\partial g}{\partial x}
$$
which satisfies
$$
\{f,g\}=-\{g,f\},\{f,gh\}=\{f,g\}h+g\{f,h\} 
$$
and the Jacobi identity
$$
\{f,\{g,h\}\}=\{\{f,g\},h\}+\{g,\{f,h\}\}.
$$
Given a polynomial $H\in \C[x,y]$, one can define the corresponding Hamiltonian vector field
$$
v_H=\{H,-\}=\frac{\partial H}{\partial x}\partial_{y}-\frac{\partial H}{\partial y}\partial_{x}.
$$
The space of all Hamiltonian vector fields is denoted by $\mathcal{H}_2$. The Jacobi identity implies
$$
[v_H,v_{H'}]=v_{\{H,H'\}},
$$
so $\mathcal{H}_2$ is closed under the commutator and forms a Lie algebra. It has a basis $v_{a,b}=v_{x^ay^b}$ ($a+b\ge 1$) and the commutation relations can be written explicitly as
$$
[v_{a,b},v_{a',b'}]=(ab'-a'b)v_{a+a'-1,b+b'-1}.
$$
Here and below we assume $v_{0,0}=0$.
We will consider the following  Lie subalgebra: 
$$
\mathcal{H}_2^{\ge 2}=\langle v_{a,b}\mid a+b\ge 2\rangle\subset \mathcal{H}_2.
$$
Also note that 
$$
\langle v_{2,0},v_{1,1},v_{0,2}\rangle\simeq \mathfrak{sl}_2\subset \mathcal{H}_2.
$$
Furthermore,
$$
[v_{a,0},v_{0,b}]=abv_{a-1,b-1},
$$
and the elements $v_{a,0},v_{0,b}$ for $a,b\ge 1$ already generate $\mathcal{H}_2$ while $v_{a,0},v_{0,b}$ for $a,b\ge 2$ generate $\mathcal{H}_2^{\ge 2}$.  
The Lie algebra $\mathcal{H}_2$ acts on $\C[\xx,\yy]$ by 
$$
v_{a,b}\mapsto \sum_{i=1}^{n}\left(ax_i^{a-1}y_i^b\partial_{y_i}-bx_i^ay_i^{b-1}\partial_{x_i}\right).
$$
In particular, the operators $E_k,F_k$ from \eqref{eq: def E algebra},\eqref{eq: def F algebra} appear as
$$
E_k=-\frac{1}{k+1}v_{0,k+1},\ F_k=\frac{1}{k+1}v_{k+1,0}
$$
and thus generate an action of $\mathcal{H}_2^{\ge 2}$.

\begin{definition}
We say that $M$ is a module over $\C[x,y]$ with a compatible action of $\mathcal{H}_2$ if for all $m\in M, f\in \C[x,y]$ and $v_H\in \mathcal{H}_2$ we have
$$
v_H(fm)=v_H(f)m+fv_H(m).
$$
\end{definition}

The following is well known but we include the proof for the reader's convenience.

\begin{lemma}
\label{lem: reduced}
Suppose $M$ is a finitely generated graded module over $\C[x,y]$ with a compatible action of $\mathcal{H}_2$. Then $M=\overline{M}[x,y]$ where $\overline{M}$ is a finite-dimensional representation of $\mathcal{H}_2^{\ge 2}$.
\end{lemma}

\begin{proof}
Note that $v_{1,0}=\partial_y$ and $v_{0,1}=-\partial_x$, and $[v_{1,0},v_{0,1}]=0$. Define $\overline{M}=\Ker(v_{1,0})\cap \Ker(v_{0,1})$.

We claim that the map $\overline{M}[x,y]\to M$ is injective. Indeed, suppose $\sum f_i(x,y)m_i=0$ for some homogeneous $f_i\in \C[x,y]$ and linearly independent $m_i\in \overline{M}$. Without loss of generality $f_1$ has largest degree among $f_i$. Then applying $f_1(-v_{0,1},v_{1,0})$ we get $\sum \lambda_im_i=0$ where $\lambda_i\in \C$ and $\lambda_1\neq 0$, contradiction.

Now let $m\in M$ be the element of minimal degree not in $\overline{M}[x,y]$. Let $m_x=-v_{0,1}(m)$, $m_y=v_{1,0}(m)$, by our assumption these are contained in $\overline{M}[x,y]$. Furthermore, $v_{1,0}(m_x)=-v_{0,1}(m_y)$, which implies (since the action of $v_{1,0}$ and $v_{0,1}$ on $\overline{M}[x,y]$ is standard)  that $m_x=ax^{a-1}y^bm_0$ and $m_y=bx^ay^{b-1}m_0$ for some $m_0\in \overline{M}$. Then $m-x^ay^bm_0\in \Ker(v_{1,0})\cap \Ker(v_{0,1})$, contradiction.
This completes the proof of the isomorphism $\overline{M}[x,y]$.

Next, we claim that the $\C[x,y]$-submodule $(x,y)\overline{M}\subset \overline{M}[x,y]$ is invariant under the action of $\mathcal{H}_2^{\ge 2}$. Indeed, for $i+j\ge 1$ and $a+b\ge 2$ we have
$$
v_{a,b}(x^iy^jm)=(aj-bi)x^{a+i-1}y^{b+j-1}m+x^iy^jv_{a,b}(m).
$$
Since $(a+i-1)+(b+j-1)\ge 2+1-2=1$, both terms are contained in $(x,y)\overline{M}$. 

Therefore we can define the action of $\mathcal{H}_2^{\ge 2}$ on $\overline{M}$ by writing $\overline{M}=\overline{M}[x,y]/(x,y)\overline{M}$.
\end{proof}

\begin{remark}
The action of $\mathcal{H}_2^{\ge 2}$ on $\overline{M}$ does not extend to the action of $\mathcal{H}_2$. Indeed, $v_{1,0}$ acts by zero on $\overline{M}$ while $[v_{1,0},v_{a,b}]=bv_{a,b-1}$ in $\mathcal{H}_2$.
This is also clear from the proof since $v_{1,0}$ does not preserve the submodule  $(x,y)\overline{M}$.
\end{remark}

\section{Homology of torus links}
\label{sec:links}

\subsection{Full twist revisited}

Let us recall the description of the $y$-fied homology of the full twist following \cite{GH}. Let $\FT_n=(\sigma_1\cdots\sigma_{n-1})^{n}$ be the full twist braid on $n$ strands and $\bFT_n=T^y(\FT_n)$ the corresponding $y$-ified Rouquier complex. The closure of $\FT_n$ is the torus link $T(n,n)$.

\begin{theorem}(\cite{GH})
\label{thm: GH}
There is a chain map $\Psi:\bFT_n\to \one_n^y$ which induces an injective map $\HY(\Psi):\HY(\FT_n)\to \HY(\one_n)=\C[\xx,\yy]\otimes \wedge(\ttheta)$. The image of $\HY(\Psi)$ is the ideal $\cJ\subset \C[\xx,\yy]\otimes \wedge(\ttheta)$ generated by antisymmetric polynomials.
Furthermore, we have an isomorphism of $\C[\xx]$-modules
$$
\HHH(\FT_n)=\HY(\FT_n)/(\yy)\HY(\FT_n)\simeq \cJ/(\yy)\cJ.
$$
\end{theorem}

In particular, in $a$-degree zero we have $$
\HY^0(\FT_n)\simeq J,\quad \HHH^0(\FT_n)\simeq J/(\yy)J.
$$
The following result computes the action of tautological classes $\CF_k$ and $\CE_k$ on $\HY^0(\FT_n)$.

\begin{theorem}
\label{thm: tautol FT}
The action of $\CE_k,\CF_k$ on $\HY^0(\FT_n)$ is given by \eqref{eq: def E algebra} and \eqref{eq: def F algebra} under the isomorphism $\HY^0(\FT_n)\simeq J$. The action of the involution $\Phi$ on $\HY^0(\FT_n)$ corresponds to exchanging $\xx$ with $\yy$ in $J$.
\end{theorem}

\begin{proof}
Recall that the map $\Psi:\bFT_n\to \one_n^y$ was constructed in \cite{GH} by composing $\binom{n}{2}$ maps  from Proposition \ref{prop: skein} changing positive crossings to negative. By Theorem \ref{thm: Fk summary}(d), all these maps commute with the action of $\CF_k$ and hence $\Psi$ commutes with the action of $\CF_k$ as well.  By Corollary \ref{cor: hard Lefshetz} $\Psi$ commutes with the action of the involution $\Phi$ and of $\CE_k$.

It remains to describe the action of $\CF_k,\Phi$ and $\CE_k$ on $\HY(\one_n)$. It follows from Theorem \ref{thm: Fk summary}(c) that the action of $\CF_k$ is given by \eqref{eq: def F algebra}. By \eqref{eq: def Phi} $\Phi$ exchanges $x_i$ with $y_i$ and hence the action of $\CE_k$ is given by \eqref{eq: def E algebra}.
\end{proof}

Let $\delta\in \HY^0(\FT_n)$ be the  class corresponding to the unique copy of $R$ on the right of the $y$-ified Rouquier complex $\bFT_n$. 

\begin{lemma}
\label{lem: delta}
Under the isomorphism $\HY^0(\FT_n)\simeq J$ the class $\delta$ corresponds to $\Delta(\yy)$.  
\end{lemma}

\begin{proof}
This follows from \cite[Lemma 5.9]{BGHW}.
\end{proof}

\subsection{From the full twist to $T(n,n+1)$}

Next, we study the $(n,n+1)$ torus braid 
$$
\beta_{n,n+1}=(\sigma_1\cdots\sigma_{n-1})^{n+1}=\FT_n(\sigma_1\cdots\sigma_{n-1}).
$$
This is no longer a pure braid, so we need to be careful when computing its $y$-ified homology.

\begin{lemma}
We have homotopy equivalences
$$
T_1^y\otimes \cdots\otimes  T_{n-1}^y\otimes K_{n-1}^y\otimes \cdots \otimes K_1^y\simeq T_1^y\otimes K_1^y\otimes \cdots \otimes T_{n-1}^y\otimes K_{n-1}^y
$$
which agree with the action of tautological classes.
\end{lemma}

\begin{proof}
By \cite[Proposition 4.11]{GHM} $T_{n-1}^y\otimes K_{n-1}^y$ commutes with $K_i^y$ for all $i$ and the corresponding homotopy equivalence agrees with the action of tautological classes. Therefore
$$
T_1^y\otimes \cdots\otimes  T_{n-1}^y\otimes K_{n-1}^y\otimes \cdots \otimes K_1^y\simeq 
T_1^y\otimes \cdots\otimes  (K_{n-2}\otimes \cdots \otimes K_1^y)\otimes (T_{n-1}^y\otimes K_{n-1}^y)
$$
and we can proceed by induction.
\end{proof}

\begin{corollary}
\label{cor: HY n n+1}
We have 
$$
\HY(\beta_{n,n+1})=\HY(\bFT_n\otimes T_1^y\otimes K_1^y\otimes \cdots \otimes T_{n-1}^y\otimes K_{n-1}^y).
$$
\end{corollary}

Next, we would like to construct a distinguished chain map from $\HY(\FT_n)$ to $\HY(\beta_{n,n+1})$ corresponding to adding crossings. 

\begin{lemma}
\label{lem: pi i}
There is a chain map $\pi_i:\one_n^y\to T_i^y\otimes K_i^y$ which  commutes with the action of tautological classes. Furthermore, $(x_i-x_{i+1})\pi_i$ and $(y_i-y_{i+1})\pi_i$ are null-homotopic. 
\end{lemma}

\begin{proof}
The map $\pi_i$ is a part of skein exact triangle from Lemma \ref{prop: skein} connecting $\one_n^y$, $T_i^y\otimes K_i^y$ and $(T_i^y)^2$, in particular, it commutes with the action of tautological classes by Theorem \ref{thm: Fk summary}(d). For the reader's convenience, we write the map explicitly:

\begin{center}
\begin{tikzcd}
 & &  & & R[\yy] \arrow[bend right]{dllll}{b^*} \arrow[bend left]{dddll}{1}\\
B_i[\yy] \arrow[bend left]{rr}{b} \arrow{dd}{y}& & R[\yy] \arrow{ll}{yb^*} \arrow[bend left]{dd}{y}& & \\
 & & & & \\
 B_i[\yy] \arrow {rr}{b} \arrow[bend left]{uu}{x}& & R[\yy] \arrow[bend left]{ll}{yb^*} \arrow{uu}{x} & & 
\end{tikzcd}
\end{center}
Here we abbreviate $b=b_i,b^*=b_i^*$, $x=x_i-x'_{i+1}$ and $y=y_i-y_{i+1}$. 
The homotopy for $(y_i-y_{i+1})\pi_i$ sends $R$ to the top copy of $R$, while the homotopy for $(x_i-x_{i+1})\pi_i$ sends $R$ to the bottom copy of $B_i$.
\end{proof}

\begin{remark}
One can check that $\pi_i$ can be interpreted as a $y$-ification of the Elias-Krasner cobordism map in \cite{EK}. Note that the crossing $\sigma_i$ closes up to the stabilized unlink, so
$$
\Hom(\one,T_i)=\HHH^0(\sigma_i)=\C[\xx]/(x_i-x_{i+1})
$$
while
$$
\Hom^y(\one^y,T_i^y\otimes K_i^y)=\HY^0(\sigma_i)=\C[\xx,\yy]/(x_i-x_{i+1},y_i-y_{i+1})
$$
and the map $\pi_i$ generates the latter morphism space while the Elias-Krasner map generates the former.
\end{remark}

\begin{lemma}
\label{lem: pi}
There is a degree zero chain map $\pi_y:\HY(\FT_n)\to \HY(\beta_{n,n+1})$  which commutes with the action of tautological classes and sends the class $\delta\in \HY^0(\FT_n)$ to a nonzero homology class. Furthermore,
$(x_i-x_j)\pi_y=(y_i-y_j)\pi_y=0$ for all $i$ and $j$.
\end{lemma}

\begin{proof}
We get $\pi_y$ by composing the maps $\pi_i$ from Lemma \ref{lem: pi i} for all $i$ and applying Corollary \ref{cor: HY n n+1}. Let us prove that $\pi_y(\delta)$ is nonzero. Indeed, the chain complex $\bFT_n\otimes T_1^y\otimes K_1^y\otimes \cdots \otimes T_{n-1}^y\otimes K_{n-1}^y$ contains several copies of $R$ connected by various differentials. By the above, the class $\pi_y(\delta)$ is a closed element of this complex which has a component in one of these copies of $R$ of degree $0$. On the other hand, all differentials in the Koszul complexes $K_i^y$ have nonzero $(q,t)$ degree, and after applying $\HH$ all maps $\HH(B_i)\to \HH(R)$  also land in nonzero degrees. Therefore $\pi_y(\delta)$ cannot be a boundary and hence represents a nonzero homology class.

Next, we check that $\pi_y$ has degree zero. This can be done explicitly by tracking all grading conventions. However, we can shortcut by considering an analogous map 
\begin{equation}
\label{eq: stabilized unknot}
\pi_{\one,y}:\HY^0(\one_n)\to \HY^0(T_1^y\otimes K_1^y\otimes \cdots \otimes T_{n-1}^y\otimes K_{n-1}^y).
\end{equation}
Clearly, $\pi_{\one,y}$ and $\pi_y$ have the same degree.
Similarly to the above, the map $\pi_{\one,y}$ is nonzero. However, the right hand side of \eqref{eq: stabilized unknot} computes the $y$-ified homology of the stabilized unknot which equals $\HY^0(\one_n)/(x_i-x_j,y_i-y_j)$ with no shift, so $\pi_{\one,y}$ (and hence $\pi_y$) has degree zero.

All other properties of $\pi_y$ follow from Lemma \ref{lem: pi i}.
\end{proof}

\begin{corollary}
\label{cor:pi bar}
There is a chain map $\overline{\pi}_y:\HY^0(\FT_n)\to \overline{\HHH}^0(\beta_{n,n+1})$  which commutes with the action of tautological classes and sends the class $\delta\in \HY^0(\FT_n)$ to a nonzero homology class. Furthermore,
$x_i\overline{\pi}_y=y_i\overline{\pi}_y=0$ for all $i$.
\end{corollary}

\begin{proof}
This follows from Lemma \ref{lem: pi} and decomposition
$$
\HY^0(\beta_{n,n+1})=\overline{\HHH}^0(\beta_{n,n+1})\otimes \C[x_1+\ldots+x_n,y_1+\ldots+y_n].
$$
\end{proof}

Next, we recall the following result conjectured by the first author in \cite{G} and proved in \cite{H}. 

\begin{theorem}[\cite{H}]
\label{thm: dimension}
The bigraded Hilbert series of $\overline{\HHH}^0(\beta_{n,n+1})$ is given by the $q,t$-Catalan number $c_n(q,t)$ and agrees with the one of $\DR_n^{\sgn}$. In particular, the total dimension of $\overline{\HHH}^0(\beta_{n,n+1})$ equals $c_n$.
\end{theorem}

\begin{theorem}
\label{thm: final}
The map $\overline{\pi}_y: \HY^0(\FT_n)\to \overline{\HHH}^0(\beta_{n,n+1})$ is surjective. We have a commutative diagram where the vertical arrows are isomorphisms commuting with the action of tautological classes:
\begin{center}
\begin{tikzcd}
\HY^0(\FT_n)
\arrow{d}{\simeq} \arrow{r}{\overline{\pi}_y}& \overline{\HHH}^0(\beta_{n,n+1}) \arrow{d}{\simeq} \\
J \arrow{r} & J/\mm J.
\end{tikzcd}
\end{center}
\end{theorem}

\begin{proof}
By Corollary \ref{cor:pi bar} the map $\overline{\pi}_y$ factors through the quotient:
$$
\HY^0(\FT_n)\twoheadrightarrow \HY^0(\FT_n)/\mm \HY^0(\FT_n) \xrightarrow{\pi'} \overline{\HHH}^0(\beta_{n,n+1}).
$$
We can rewrite it as a map
$$
\pi':J/\mm J\to  \overline{\HHH}^0(\beta_{n,n+1}).
$$
Let us prove that $\pi'$ is injective. By Lemma \ref{lem: delta} we have $\pi'(\Delta(\yy))=\overline{\pi}_y(\delta)\neq 0$.

Now we identify $J/\mm J\simeq \DR_n^{\sgn}$ and use Lemma \ref{lem: cogenerate}. Suppose that $f\in J/\mm J$ is a nonzero class. Then there exists a polynomial $\psi(E_1,\ldots,E_{n-1})$ such that  $\psi(E_1,\ldots,E_{n-1})f=c\Delta(\yy)$ for $c\neq 0$. Now 
$$
\psi(\CE_1,\ldots,\CE_{n-1})\pi'(f)=\pi'\left[\psi(E_1,\ldots,E_{n-1})f\right]=\pi'(c\Delta(\yy))=c\overline{\pi}_y(\delta)\neq 0.
$$
Therefore $\pi'(f)\neq 0$ as well, and $\pi'$ is injective. On the other hand, by Theorem \ref{thm: dimension} the total dimensions of $J/\mm J$ and $\overline{\HHH}^0(\beta_{n,n+1})$ agree, hence $\pi'$ is an isomorphism.
\end{proof}

\begin{corollary}
The maps $\pi_y:\HY^0(\FT_n)\to \HY^0(\beta_{n,n+1})$ and $\pi:\HHH^0(\FT_n)\to \HHH^0(\beta_{n,n+1})$ are surjective.
\end{corollary}

\begin{proof}
By Theorem \ref{thm: final} we can write
$$
\HY^0(\beta_{n,n+1})\simeq J/(x_i-x_j,y_i-y_j)J,\ \HHH^0(\beta_{n,n+1})\simeq J/(x_i-x_j,\yy)J.
$$
This proves the first statement since $\pi_y$ can be identified with the projection
$$
J\twoheadrightarrow J/(x_i-x_j,y_i-y_j)J.
$$
For the second statement, we use Theorem \ref{thm: GH} and identify the map $\pi$ with the projection
$$
\HHH^0(\FT_n)=J/(\yy)J\twoheadrightarrow J/(x_i-x_j,\yy)J.
$$
\end{proof}

\subsection{Hooks and higher $a$-degrees}
\label{sec: hooks}

In this section we  generalize Theorem \ref{thm: final} to higher $a$-degrees. First, we need some notations. Let $V=\C^{n-1}$ be the $(n-1)$-dimensional reflection representation of $S_n$, define $V_i=\wedge^{i}V,\ 0\le i\le n-1$. It is known that $V_i$ are irreducible representations of $S_n$ labeled by hook-shaped partitions $(n-i,1^{i})$, and $\wedge^{n-1-i}V=\wedge^iV\otimes \sgn$.

\begin{definition}
If $W$ is a representation of $S_n$, we define 
$$
W^{\hook}=\bigoplus_{i=0}^{n-1}\Hom_{S_n}(\wedge^{n-1-i}V,W).
$$
\end{definition}
By Schur's lemma we get
$$
W^{\hook}=\left[\bigoplus_{i=0}^{n-1}\wedge^{n-1-i}V\otimes W\right]^{S_n}=\left[\bigoplus_{i=0}^{n-1}\wedge^{i}V\otimes W\right]^{\sgn}=\left[\wedge^{\bullet}V\otimes W\right]^{\sgn}.
$$
It will be useful to realize $V$ explicitly as the quotient of the $n$-dimensional polynomial representation:
$$
\langle \theta_1,\ldots,\theta_n\rangle = V\oplus \langle \theta_1+\ldots+\theta_n\rangle,
$$
so that 
\begin{equation}
\label{eq: reduced wedge}
\wedge(\ttheta)=\wedge^{\bullet}\langle \theta_1,\ldots,\theta_n\rangle=\wedge^{\bullet} V\otimes \wedge^{\bullet}\langle\theta_1+\ldots+\theta_n\rangle.
\end{equation}

\begin{lemma}
\label{lem: solomon}
Let $\omega_N=\sum_{i=1}^nx_i^{N}\theta_i$. Then
$$
\left[\wedge(\ttheta)\otimes \C[\xx]\right]^{S_n}=\wedge^{\bullet}\langle 
\omega_0,\ldots,\omega_{n-1}\rangle \otimes \C[\xx]^{S_n},
$$
$$
\left[\wedge^{\bullet}V\otimes \C[\xx]\right]^{S_n}=\wedge^{\bullet}\langle 
\omega_1,\ldots,\omega_{n-1}\rangle \otimes \C[\xx]^{S_n},
$$
\end{lemma}

\begin{proof}
We can identify $\wedge(\ttheta)\otimes \C[\xx]$ with the space $\Omega(\C^n)$ of polynomial differential forms on $\C^n$ by writing $\theta_i=dx_i$. By a result of Solomon \cite{Solomon} we have
$$
\Omega(\C^n)^{S_n}\simeq \Omega(\C^n/S_n),
$$
so the $S_n$-invariant differential forms are freely generated by power sums $p_k=\sum_{i=1}^{k}x_i^k,\ 1\le k\le n$ and their differentials
$$
dp_k=k\sum_{i=1}^{k}x_i^{k-1}dx_i=k\omega_{i-1}.
$$
The last statement follows from \eqref{eq: reduced wedge} since $\omega_0=\theta_1+\ldots+\theta_n$.
\end{proof}

\begin{definition}
We define the operators $\theta_i^*$ on $\wedge(\ttheta)$ by $\theta_i^*(\theta_j)=\delta_{ij}$ and extending by Leibniz rule.

We define the operators $d^*_N,d_N$ on $\wedge(\ttheta)\otimes \C[\xx]$ as follows:
\begin{equation}
\label{eq: dN algebra}
d_N=\sum_{i=1}^{n}\theta_i^*x_i^{N},\ d_N^*=\sum_{i=1}^{n}\theta_i\partial_{x_i}^{N},\ N\ge 0.
\end{equation}
For $N>0$ we can use \eqref{eq: reduced wedge} to interpret $d_N$ and $d_N^*$ as operators on
$\wedge^{\bullet}V\otimes \C[\xx]$.
\end{definition}

\begin{lemma}
\label{lem: hooks generate}
a) The space 
$$
\DH_n^{\hook}\simeq \left[\wedge^{\bullet}V\otimes \DH_n\right]^{\sgn}.
$$
is generated by $\Delta(\xx)$ under the action of $\C[F_1^*,\ldots,F_{n-1}^*]\otimes \wedge^{\bullet}\langle d_{1}^*,\ldots,d_{n-1}^*\rangle$.

b) For any nonzero homogeneous element 
$$
f\in \DR_n^{\hook}\simeq \left[\wedge^{\bullet}V\otimes \DR_n\right]^{\sgn}
$$
there exists a polynomial $\varphi\in \C[F_1,\ldots,F_{n-1}]\otimes \wedge^{\bullet}\langle d_{1},\ldots,d_{n-1}\rangle$ such that $\varphi(f)$ is a nonzero multiple of $\Delta(\xx)$.
\end{lemma}

\begin{proof}
By Theorem \ref{thm: operator conjecture} the element $\Delta(\xx)$ generates $\wedge^{\bullet}V\otimes \DH_n$ under the action of 
$$
\wedge^{\bullet}(V)\otimes \C[F_1^*,\ldots,F_{n-1}^*,\partial_{x_1},\ldots,\partial_{x_n}].
$$
Similarly to Corollary \ref{cor: operator conjecture sgn}, we deduce that $\Delta(\xx)$  generates 
$
\left[\wedge^{\bullet}V\otimes \DH_n\right]^{\sgn}
$
under the action of 
$$
\left[\wedge^{\bullet}(V)\otimes \C[F_1^*,\ldots,F_{n-1}^*,\partial_{x_1},\ldots,\partial_{x_n}]\right]^{S_n}=\C[F_1^*,\ldots,F_{n-1}^*]\otimes \left[\wedge^{\bullet}(V)\otimes\C[\partial_{x_1},\ldots,\partial_{x_n}]\right]^{S_n}=
$$
$$
\C[F_1^*,\ldots,F_{n-1}^*]\otimes \C[\partial_{x_1},\ldots,\partial_{x_n}]^{S_n}\otimes \wedge^{\bullet}\langle d_1^*,\ldots,d_{n-1}^*\rangle.
$$
The last equation follows from Lemma \ref{lem: solomon}. Since any homogeneous polynomial of positive degree in $\C[\partial_{x_1},\ldots,\partial_{x_n}]^{S_n}$ annihilates $\Delta(\xx)$, part (a) follows.

Part (b) is dual to (a) and the proof is similar to Lemma \ref{lem: cogenerate}.
\end{proof}

\begin{definition}
Let $\cJ$ be the ideal in $\C[\xx,\yy]\otimes \wedge(\ttheta)$ generated by antisymmetric polynomials and let $\overline{\cJ}=\cJ/\langle \theta_1+\ldots+\theta_n\rangle \cJ\subset \wedge^{\bullet}V\otimes \C[\xx,\yy]$.
\end{definition}

\begin{lemma}
\label{lem: iota}
We have an isomorphism
$$
\overline{\cJ}/\mm\overline{\cJ}\simeq \DR_n^{\hook}.
$$
\end{lemma}

\begin{proof}
We follow the proof of Lemma \ref{lem: catalan}.
Let $\cL\subset \wedge^{\bullet}V\otimes \C[\xx,\yy]$ be the subspace spanned by the products $f(\xx,\yy)g(\xx,\yy,\ttheta)$ where $f$ and $g$ are homogeneous, $f$ has positive degree and symmetric, and $g$ is antisymmetric. 
Let us prove that 
$$
\left[\C[\xx,\yy]^{S_n}_{+}\right]^{\hook}\simeq\cL.
$$
Indeed, we get
$$
\left[\C[\xx,\yy]^{S_n}_{+}\right]^{\hook} =\left[\wedge^{\bullet}V\otimes \C[\xx,\yy]^{S_n}_{+}\right]^{\sgn}.
$$
The space $\wedge^{\bullet}V\otimes \C[\xx,\yy]^{S_n}_{+}$ is spanned by the sums $\sum f_i(\xx,\yy)h_i(\xx,\yy,\ttheta)$ where $f_i(\xx,\yy)$ are symmetric polynomials of positive degree, and $h_i$ are arbitrary. Similarly to the proof of Lemma \ref{lem: catalan}, by antisymmetrizing such a sum we get an element in $\cL$.

By \cite[Lemma 7.5]{GH} the ideal $\cJ$ (and hence $\overline{\cJ}$) is in fact generated by antisymmetric polynomials over $\C[\xx,\yy]$. Therefore any element of $\mm\cJ$ can be written as $\sum g_i(\xx,\yy,\ttheta)h_i(\xx,\yy)$ where $g_i$ are antisymmetric and $h_i$ have positive degree, in particular $\cJ/\mm\cJ$ is spanned by antisymmetric polynomials. Similarly to the proof of Lemma \ref{lem: catalan}, we get 
$\cL=\left[\mm\overline{\cJ}\right]^{\sgn}$. 
Also, note that $$\overline{\cJ}^{\sgn}=\left[\wedge^{\bullet}V\otimes \C[\xx,\yy]\right]^{\sgn}\simeq \C[\xx,\yy]^{\hook}.
$$ 
  
Finally, we get an isomorphism
$$
\DR_n^{\hook} =
\frac{\left[\C[\xx,\yy]\right]^{\hook}}{\left[\C[\xx,\yy]^{S_n}_{+}\right]^{\hook}}=\frac{\left[\wedge^{\bullet}V\otimes \C[\xx,\yy]\right]^{\sgn}}{\cL}= 
\left[\overline{\cJ}/\mm\overline{\cJ}\right]^{\sgn}=\overline{\cJ}/\mm\overline{\cJ}.
$$

\end{proof}

\begin{lemma}
The action of $\CF_k,\CE_k$ on $\HY(\FT_n)$ 
agrees with the action of $F_k,E_k$ on $\cJ$ under the isomorphism from Theorem \ref{thm: GH}.

The action of  differentials $d_N$ on $\HY(\FT_n)$ agrees with the action of operators \eqref{eq: dN algebra} on $\cJ$.
\end{lemma}

\begin{proof}
The proof of the first statement is similar to the proof of Theorem \ref{thm: tautol FT}. 

For the second statement, observe that Rasmussen differentials $d_N$ are given by \eqref{eq: dN algebra} in $\HY(\one_n)$, and the result follows from Proposition \ref{prop: dN functorial}.
\end{proof}

Now we are ready to prove a generalization of Theorem \ref{thm: final} to all $a$-degrees.

\begin{theorem}
\label{thm: final full}
The map $\overline{\pi}_y:\HY(\FT_n)\to \overline{\HHH}(\beta_{n,n+1})$ is surjective. Furthermore, we have an isomorhism
$$
\DR_n^{\hook}\simeq \overline{\HHH}(\beta_{n,n+1})
$$
which agrees with the action of tautological classes and differentials $d_N$.
\end{theorem}

\begin{proof}
Recall that 
$$
\HY(\beta_{n,n+1})=\overline{\HHH}(\beta_{n,n+1})\otimes \C[x_1+\ldots+x_n,y_1+\ldots+y_n]\otimes \wedge\langle \theta_1+\ldots+\theta_n\rangle.
$$
We use the map $\pi_y:\HY(\FT_n)\to \HY(\beta_{n,n+1})$, by Lemma \ref{lem: pi} and Proposition \ref{prop: dN functorial} it commutes with the action of tautological classes and differentials.

By Lemma \ref{lem: pi} the map $\overline{\pi}_y:\HY(\FT_n)\to \overline{\HHH}(\beta_{n,n+1})$ factors through $\overline{\cJ}/\mm \overline{\cJ}$ and we get the following diagram:
$$
\begin{tikzcd}
\HY(\FT_n)\simeq \cJ \arrow[twoheadrightarrow]{r} \arrow[bend left,twoheadrightarrow]{rr}{\overline{\pi}_y}& \overline{\cJ}/\mm \overline{\cJ} \arrow{r}{\pi'_y}&  \overline{\HHH}(\beta_{n,n+1})\\
 & \DR_n^{\hook} \arrow{u}{\iota}  \arrow{ur}{\simeq} & .
\end{tikzcd}
$$
The vertical map $\iota$ is the isomorphism from Lemma \ref{lem: iota}. 

Let us prove that the composition $\pi'_y\circ \iota$ is injective. Indeed, by Theorem \ref{thm: final} we have $\pi'_y(\iota(\Delta(\xx)))\neq 0$. Let $f\in \DR_n^{\hook}$ be a nonzero homogeneous polynomial, then by Lemma \ref{lem: hooks generate} there exists a polynomial $\varphi\in \C[F_1,\ldots,F_{n-1}]\otimes \wedge^{\bullet}\langle d_1,\ldots,d_{n-1}\rangle$ such that $\varphi(f)=c\Delta(\xx)$ for $c\neq 0$. Now
$$
\varphi(\pi'_y(\iota(f)))=\pi'_y(\iota(\varphi(f)))=c\pi'_y(\iota(\Delta(\xx)))\neq 0,
$$
hence $\pi'_y(\iota(f))\neq 0$. 

Finally, by the main result of \cite{H} we have $$\dim \DR_n^{\hook}=\dim \overline{\HHH}(\beta_{n,n+1}),$$ and $\pi'_y\circ \iota$ is an isomorphism.
\end{proof}

\end{document}